\documentclass[oneside,english]{amsart}
\usepackage[latin9]{inputenc}
\usepackage{amstext}
\usepackage{amsthm}
\usepackage{amssymb}

\makeatletter
\numberwithin{equation}{section}
\numberwithin{figure}{section}
\theoremstyle{plain}
\newtheorem{thm}{\protect\theoremname}
\theoremstyle{plain}
\newtheorem{cor}[thm]{\protect\corollaryname}

\makeatother

\usepackage{babel}
\providecommand{\corollaryname}{Corollary}
\providecommand{\theoremname}{Theorem}

\begin{document}
\title{SETS WITH ARITHMETIC PROGRESSIONS ARE ABUNDANT}
\author{ANINDA CHAKRABORTY AND SAYAN GOSWAMI}

\email{anindachakraborty2@gmail.com}
\address{Government General Degree College at Chapra}
\email{sayan92m@gmail.com}
\address{Department of Mathematics,University of Kalyani}
\thanks{The second author of the paper is supported by UGC-JRF fellowship.}
\begin{abstract}
{\normalsize{}Furstenberg, Glasscock, Bergelson, Beiglboeck have been
studied abundance in arithmatic progression on various large sets
like piecewise syndetic, central, thick, etc. but also there are so
many sets in which abundance in progression is still unsettled like
J-sets, C-sets, D-sets etc. But all of these sets have a common property
that they contains arbitrary length of arithmatic progressions. These
type of sets are called sets of A.P. rich, we have given an elementary
proof of abundance of those sets.}{\normalsize\par}
\end{abstract}

\maketitle

\section{introduction}

One of the famous Ramsey theoretic results is so called Van der Waerden\textquoteright s
Theorem which guarantees that atleast one cell of any partition $\{C_{1},C_{2},\ldots,C_{r}\}$
of $\mathbb{N}$ contains arithmetic progressions of arbitrary length.
Since arithmetic progressions are invariant under shifts, it follows
that every piecewise syndetic set contains arbitrarily long arithmetic
progressions.
\begin{thm}
\label{Thm 1}Given any $r,l\in\mathbb{N}$, there exists $N(r,l)\in\mathbb{N}$,
such that for any $r-$partition of $[1,N]$, atleast one of the partition
contains an $l-$length arithmetic progressions.
\end{thm}

In various times, mathematicians studied abundance in progression
in different types of large sets Like \textbf{syndetic sets, central
sets, thick sets, piecewise syndetic sets }etc. we have seen many
abundance results from \cite{key-4}, \cite{key-5}, \cite{key-6},
\cite{key-3} etc. All of this results shows that if $A\subseteq\mathbb{N}\text{ or \ensuremath{S}}$,
(Where S is any countable commutative semigroup) be large in some
sense then some special configuration contained in those sets are
also large in some sense. However, there also remains types of large
sets where abundance are yet to be explicate, like \textbf{C-sets,
D-sets, J-sets }etc.

All of these aforementioned sets have a common property: They all
contain arbitrary length of arithmatic progressions, this type of
sets are called \textbf{sets of A.P. rich}. Here we have given easiest
elementary combinatorial proof of abundance for these type of sets.
Also we have seen that if $A$ is a \textbf{set of A.P. rich}, then
it is \textbf{set of A.P. rich} of all order.

Throughout this paper, $S$ is a countable commutative semigroup.
Although sometimes countability or commutativity does not appear in
the proof.

\section{Main results}
\begin{thm}
\label{Thm 2} Let $S$ is any semigroup. Then If $A\subseteq S$
is a set of A.P. rich, then 
\[
B=\:\{(a,b):\:\left\{ a,\:a+b,\:a+2b,....,\:a+lb\right\} \subset A\}
\]
is a set of A.P. rich.
\end{thm}

\begin{proof}
Now as $A$ contains arithmatic progression of arbitrary length, fixed
$l\in\mathbb{N}$ and for any $1\leq r\leq l$, it must contains arithmatic
progression of length $r+(r+1)l$, which implies that $(c,d)+r(d,d)\in B$,
which shows that $B$ contains 
\[
\left\{ (c,d),\:(c,d)+(d,d),\:(c,d)+2(d,d),.....,\:(c,d)+l(d,d)\right\} .
\]
\end{proof}
This proves the theorem.

Suppose, for any $l\in\mathbb{N}$, $AP_{l+1}$ be the subsemigroup
of $S^{l+1}$ defined by:
\[
AP_{l+1=\:}\left\{ \left(a,\:a+b,\:a+2b,.....,\:a+lb\right):\:a,b\in S\right\} .
\]

Now we derive a result which is one of the main application of \cite{key-3}
for some large sets.
\begin{cor}
\label{Corollary 3} Let $S$ be a cancellative semigroup and $A\subseteq S$
is a set of A.P. rich, then $A^{l+1}\cap AP_{l+1}$ is also a set
of A.P. rich.
\end{cor}

\begin{proof}
Take the semigroup epimorphism, $\varphi:\:S\times S\longrightarrow AP_{l+1}$
defined by, $\varphi\left(a,b\right)=\:\left(a,\:a+b,\:a+2b,......,\:a+lb\right)$.
Then for any $A\subseteq S$ is A.P. rich, from \ref{Thm 2} $B$
is also a set of A.P. rich. So, for any $l-$length arithmatic progression
\[
\left\{ \left(a,b\right),\:\left(a,b\right)+\left(e,f\right),\:\left(a,b\right)+2\left(e,f\right),......,\:\left(a,b\right)+l\left(e,f\right)\right\} 
\]
 in $B$, we have 
\[
\varphi\left\{ \left(a,b\right)+i\left(e,f\right)\right\} =\:\varphi\left(a,b\right)+\:i\varphi\left(e,f\right)\in A^{l+1}\cap AP_{l+1}\text{ for all \ensuremath{1\leq i\leq l}}\text{.}
\]

Which concludes the result.
\end{proof}


\begin{thebibliography}{Bel}
{\normalsize{}\bibitem[Bel]{key-1} Mathias Beiglboeck, Arithmetic
Progressions In Abundance By Combinatorial Tools, Proc. Amer. Math.
Soc. 137 (2009), no. 12, 3981-3983.}{\normalsize\par}

{\normalsize{}\bibitem[BG]{key-2} V. Bergelson, D. Glasscock,On the
interplay between additive and multiplicative largeness and its combinatorial
applications, arXiv:1610.09771 }{\normalsize\par}

{\normalsize{}\bibitem[BH]{key-3} V.Bergelson, N.Hindman. Partition
regular structures contained in large sets are abundant. J. Combin.
Theory ser. A, 93(1): 18-36, 2001.}{\normalsize\par}

{\normalsize{}\bibitem[FG]{key-4} H.Furstenberg, E. Glasner. Subset
dynamics and van der warden's theorem. In topological dynamics and
applications ( Minneapolis, MN, 1995 ), volume 215 of contemp. math.,
pages 197-203. Amer. Math. Soc, Providence, RI, 1998.}{\normalsize\par}

{\normalsize{}\bibitem[GJ]{key-5} S. Goswami, S.Jana, Abundance in
commutative semigroup, arXiv:1902.03557}{\normalsize\par}

{\normalsize{}\bibitem[HS]{key-6} N.Hindman, D.Strauss, Algebra in
the stone-\v{c}ech Compactification: theory and applications, second
edition, de Gruyter, Berlin, 2012.}{\normalsize\par}
\end{thebibliography}
\end{document}